\documentclass[11pt]{article}

\usepackage[margin=1.1in]{geometry}
\usepackage{amsmath,amssymb,amsthm}
\usepackage{xcolor}
\usepackage[colorlinks=true,linkcolor=blue,citecolor=blue,urlcolor=blue]{hyperref}

\newtheorem{theorem}{Theorem}
\newtheorem{lemma}[theorem]{Lemma}

\newtheorem{corollary}[theorem]{Corollary}

\theoremstyle{remark}
\newtheorem{remark}[theorem]{Remark}

\newcommand{\fas}{\mathrm{fas}}
\newcommand{\fasw}{\mathrm{fas}_w}
\newcommand{\fwd}{\mathrm{fwd}}
\newcommand{\bck}{\mathrm{back}}
\newcommand{\E}{\mathbb{E}}
\newcommand{\R}{\mathbb{R}}
\newcommand{\Prob}{\mathbb{P}}
\newcommand{\norm}[1]{\left\lVert #1 \right\rVert}

\newcommand{\ind}[1]{\mathbf{1}\!\left[#1\right]}

\title{The optimal constant for minimum weight feedback arc sets\\ in oriented graphs}
\author{Yacong Zhou\thanks{Shenzhen Institutes of Advanced Technology, Chinese Academy of Sciences. Emails: {\tt yacong.zhou96@gmail.com}}}
\date{\today}

\begin{document}
	\maketitle
	
	\begin{abstract}
		Let $D$ be an oriented graph (a digraph with no directed $2$-cycles) with
		maximum degree $\Delta\ge 1$, equipped with nonnegative arc weights of total
		weight $w(D)$, and let $\fasw(D)$ denote the minimum weight of a feedback
		arc set of $D$. Alon (2002) proved
		$\fasw(D)\le\bigl(\tfrac12-\tfrac{1}{16\sqrt{2\Delta}}\bigr)w(D)$. We
		determine the optimal constant:
		\[
		\fasw(D)\;\le\;\Bigl(\frac12-\frac{\sqrt2}{6\sqrt{\Delta}}\Bigr)\,w(D). 
		\]
		In fact, we show a stronger result: 
		$\fasw(D)\le\tfrac12w(D)-\tfrac{\sqrt2}{12}\sum_vw_2(v)$, where $w_2(v)$ is
		the $\ell_2$-norm of the weights of the arcs incident with $v$. Both bounds are attained by the unit-weight directed triangle, so the constant $\frac{\sqrt2}{6}$ is best possible (already among unweighted oriented graphs). The proof combines the vertex-peeling scheme of Berger and Shor with a continuous random-ordering analysis: realizing the random order by independent uniform labels renders the expected local imbalance at each vertex \emph{exactly} an integrated Khintchine-type functional, and the theorem reduces to the sharp evaluation
		\[
		\inf_{\|a\|_2=1}\int_0^1\E\,\Bigl|\sum\nolimits_ja_jB_j(q)\Bigr|\,dq
		\;=\;\frac{\sqrt2}{6},
		\qquad B_j(q)\ \text{i.i.d.\ Bernoulli}(q),
		\] 
		which we prove via Fourier analysis. The proof also yields a randomized, near-linear-time algorithm attaining the bounds in expectation.
	\end{abstract}
	
	\section{Introduction}
	
	Let $D=(V,A)$ be a digraph. A set $F\subseteq A$ is a \emph{feedback arc set}
	(FAS) if $D-F$ is acyclic. For a weight function $w\colon A\to\R_{\ge0}$ we
	write $w(D)=\sum_{a\in A}w(a)$ and let $\fasw(D)$ denote the minimum weight of
	an FAS of $D$; for unweighted digraphs the minimum size of an FAS is denoted
	$\fas(D)$. An \emph{oriented graph} (or \emph{orgraph}) is a digraph with no
	loops and no directed $2$-cycles. The \emph{degree} of a vertex $v$ is
	$d(v)=d^+(v)+d^-(v)$, and $\Delta(D)=\max_{v\in V}d(v)$; note that $\Delta$ is
	determined by the digraph alone and does not depend on the weights.
	
	Upper bounds on minimum feedback arc sets are usually studied for orgraphs,
	since the general (weighted) digraph problem reduces to the orgraph case by
	cancelling directed $2$-cycles; see, e.g., \cite{GNYZ25}. The restriction is
	also necessary for bounds of the kind considered here: a single directed
	$2$-cycle with both arcs of weight $1$ satisfies $\fasw(D)=w(D)/2$ exactly.
	
	For every ordering $\sigma=v_1v_2\cdots v_n$ of $V$, the set of
	\emph{backward arcs} (arcs $v_jv_i$ with $j>i$) is an FAS; conversely, if $F$
	is an FAS then the backward arcs of any acyclic ordering of $D-F$ form a
	subset of $F$. Hence
	\begin{equation}\label{eq:ordering}
		\fasw(D)\;=\;\min_{\sigma}\,w\bigl(\bck(\sigma)\bigr)
		\;=\;\frac{w(D)-\max_\sigma\bigl(w(\fwd(\sigma))-w(\bck(\sigma))\bigr)}{2},
	\end{equation}
	where $\fwd(\sigma)$ denotes the set of forward arcs of $\sigma$. We refer to
	$w(\fwd(\sigma))-w(\bck(\sigma))$ as the \emph{surplus} of $\sigma$.
	
	Berger and Shor \cite{BS90,BS97} proved that every unweighted orgraph $D$
	with $m$ arcs satisfies $\fas(D)\le\bigl(\tfrac12-\Omega(1/\sqrt\Delta)\bigr)m$,
	and Alon \cite{Alo02} proved the weighted bound
	\begin{equation}\label{eq:alon}
		\fasw(D)\;\le\;\Bigl(\frac12-\frac{1}{16\sqrt{2\Delta}}\Bigr)\,w(D)
	\end{equation}
	for every weighted orgraph $D$. By results of Jung \cite{Jun70} and Spencer
	\cite{Spe71} (see also de la Vega \cite{dlV83}), the order of magnitude
	$\Theta(1/\sqrt{\Delta})$ in \eqref{eq:alon} is best possible, so the
	remaining question is the optimal constant. Formally, define
	\begin{equation}\label{eq:cstar}
		c^*\;=\;\inf_{D}\ \sqrt{\Delta(D)}\,\Bigl(\frac12-\frac{\fasw(D)}{w(D)}\Bigr),
	\end{equation}
	the infimum ranging over all weighted orgraphs with $\Delta(D)\ge1$ and
	$w(D)>0$, and thus \eqref{eq:alon} states $c^*\ge\frac{1}{16\sqrt2}$.
	
	Our main result determines the optimal constant exactly. For $v\in V$ let
	\[
	w(v)\;=\;\sum_{a\ni v}w(a),
	\qquad
	w_2(v)\;=\;\Bigl(\sum_{a\ni v}w(a)^2\Bigr)^{1/2},
	\]
	the sums ranging over the arcs incident with $v$, so that
	$\sum_{v\in V}w(v)=2w(D)$ and, by the Cauchy--Schwarz inequality,
	$w_2(v)\ge w(v)/\sqrt{d(v)}$. 
	
	\begin{theorem}\label{thm:main}
		Every weighted orgraph $D$ with $\Delta=\Delta(D)\ge1$ admits an ordering
		$\sigma$ of its vertices with
		\[
		w(\fwd(\sigma))-w(\bck(\sigma))
		\;\ge\;\frac{\sqrt2}{6}\sum_{v\in V}w_2(v)
		\;\ge\;\frac{\sqrt2\,w(D)}{3\sqrt{\Delta}} .
		\]
		Consequently,
		\[
		\fasw(D)\;\le\;\frac{w(D)}{2}-\frac{\sqrt2}{12}\sum_{v\in V}w_2(v)
		\;\le\;\Bigl(\frac12-\frac{\sqrt2}{6\sqrt{\Delta}}\Bigr)\,w(D).
		\]
		All four inequalities hold with equality for the unit-weight directed triangle and therefore $c^*=\frac{\sqrt2}{6}$.
	\end{theorem}
	
	The directed triangle with unit weights has $\Delta=2$, $w=3$ and
	$\fasw=1$, whence $c^*\le\sqrt2\,(\tfrac12-\tfrac13)=\tfrac{\sqrt2}{6}$;
	the content of Theorem~\ref{thm:main} is the matching lower bound
	\begin{equation}\label{eq:window}
		c^*\;=\;\frac{\sqrt2}{6}\;=\;0.2357\ldots
	\end{equation}
	
	The proof has two ingredients. The first is the vertex-peeling scheme of
	Berger and Shor \cite{BS90,BS97}, in the form of Lemma~\ref{lem:peel}
	below: for \emph{any} processing order of the vertices, the local
	imbalances between out-weight and in-weight towards later vertices can be
	accumulated, without interference, into the surplus of a single ordering. Random orderings drive this scheme already in
	\cite{BS90,BS97}. In the rank-based analyses (see, e.g.,
	\cite[Appendix~A]{FHM24} for a recent presentation), the local
	imbalance at a vertex is a sum of \emph{negatively correlated}
	indicators --- a sample without replacement --- and is bounded from
	below in $\Omega$-form, or with explicit but lossy constants. We
	realize the random order by independent continuous labels instead:
	conditionally on the label of a vertex, its imbalance is a weighted
	sum of \emph{independent} Bernoulli variables, and its expectation
	is \emph{identically} a Khintchine-type functional of weighted
	Bernoulli sums, integrated over the Bernoulli parameter, with
	no loss. The second --- and main --- ingredient is the sharp lower bound $\bar\kappa=\frac{\sqrt2}{6}$ for this integrated functional (Theorem~\ref{thm:kbarmain}), whose extremal vectors $(1/\sqrt{2},-1/\sqrt{2})$ are realized at every vertex of the directed triangle.
	
	The paper is organized as follows. Section~\ref{sec:peel} proves the
	peeling lemma (Lemma~\ref{lem:peel}) and derives Theorem~\ref{thm:main}
	from the value $\bar\kappa=\frac{\sqrt2}{6}$. The remainder of the paper
	establishes $\bar\kappa=\frac{\sqrt2}{6}$, proceeding according to the
	shape of the unit vector $a$. Section~\ref{sec:kappa} sets up a Fourier
	framework (Lemmas~\ref{lem:known} and~\ref{lem:fourier}) and uses it to
	settle all vectors without a dominant coordinate, that is, with
	$\max_j|a_j|\le\frac1{\sqrt2}$ (Corollary~\ref{prop:nondom}). For the
	remaining \emph{dominated} vectors it introduces a standard form and,
	conditioning on the dominant coordinate, proves an exact formula for the
	integrated functional (Lemma~\ref{lem:mixedid}). This settles the case in which every other coordinate is negative
	(Theorem~\ref{thm:onesigned}) and yields an unconditional budget bound (Corollary~\ref{cor:floor}). Section~\ref{sec:residual} handles the
	\emph{residual} vectors surviving these reductions: two fluctuation
	tools --- a floor lemma (Lemma~\ref{lem:ufloor}) and a ledger
	(Lemma~\ref{lem:ledger}) --- are combined with a drift estimate into a
	three-case analysis that completes the proof of
	Theorem~\ref{thm:kbarmain}. Section~\ref{sec:conclusion} concludes with
	the pointwise version of the problem, which we leave open.
	
	\section{The ordering scheme}\label{sec:peel}
	
	Throughout, $D=(V,A,w)$ is a weighted orgraph with $n=|V|$ vertices. The following lemma repackages, in weighted form and as a statement about orderings, the vertex-processing algorithm that Berger and Shor \cite{BS90,BS97} introduced for unweighted oriented graphs.
	
	\begin{lemma}\label{lem:peel}
		Let $\pi=v_1v_2\cdots v_n$ be an arbitrary ordering of $V$. For
		$i\in\{1,\dots,n\}$ let $O_i$ (respectively $I_i$) denote the set of arcs from
		$v_i$ to (respectively into $v_i$ from) $\{v_{i+1},\dots,v_n\}$. Then there
		exists an ordering $\sigma$ of $V$ with
		\[
		w(\fwd(\sigma))-w(\bck(\sigma))\;\ge\;\sum_{i=1}^{n}\bigl|w(O_i)-w(I_i)\bigr| .
		\]
	\end{lemma}
	
	\begin{proof}
		By the definitions of $O_i$s and $I_i$s,
		$\{O_i\cup I_i\}_{i=1}^{n}$ is a partition of $A$. For each $i$, let
		$K_i\in\{O_i,I_i\}$ be a set of larger weight (if $w(O_i)=w(I_i)$ then pick $K_i$ arbitrarily from them) and
		let $\overline{K}_i$ be the other set.
		
		Define $\sigma$ as follows: first the block $F=\{v_i:K_i=O_i\}$, ordered by
		increasing index $i$, followed by the block $L=\{v_i:K_i=I_i\}$, ordered by
		\emph{decreasing} index $i$.
		
		We claim that every arc in $K=\bigcup_i K_i$ is forward with respect to $\sigma$. Let $e\in K_i$ and suppose first that $e=v_iv_j\in O_i$ with $i<j$, so $v_i\in F$.  If $v_j\in L$, then this clearly holds as the whole block $L$ is after $F$. If $v_j\in F$, then $v_j$ is ordered after $v_i$ within $F$ as the block $F$ is ordered by increasing index and $i<j$. Suppose now
		that $e=v_jv_i\in I_i$ with $j>i$, so $v_i\in L$. If $v_j\in F$, then this clearly holds as the whole block $F$ is before $L$. If $v_j\in L$, then $v_j$ is ordered before $v_i$ within
		$L$ because $L$ is ordered by decreasing index and $j>i$. In all four cases $e$ is forward, proving the claim.
		
		Consequently, as all weights are nonnegative,
		\[
		w(\fwd(\sigma))\;\ge\;\sum_i w(K_i),
		\qquad
		w(\bck(\sigma))\;\le\;w(D)-\sum_i w(K_i)\;=\;\sum_i w(\overline{K}_i),
		\]
		and therefore
		$w(\fwd(\sigma))-w(\bck(\sigma))\ge\sum_i\bigl(w(K_i)-w(\overline{K}_i)\bigr)
		=\sum_i|w(O_i)-w(I_i)|$.
	\end{proof}
	
	We now prove our main result assuming
	$\bar\kappa=\frac{\sqrt2}{6}$, where $\bar\kappa$ is the integrated Khintchine-type functional defined in \eqref{eq:kbar} below. The rest of the paper is dedicated to establishing this value. As discussed in the introduction, the novel step is how the random processing order is
	analyzed. Realizing it by independent continuous labels turns the key estimate into an \emph{identity}: the expected imbalance at a vertex $v$ equals $w_2(v)$ times the integrated functional, with no loss. A rank-based realization cannot produce such an identity, because the
	indicators involved are negatively correlated. It is this
	exactness that makes the optimal constant accessible.
	
	\begin{proof}[Proof of Theorem~\ref{thm:main}]
		Let $(U_v)_{v\in V}$ be independent Uniform $[0,1]$ random variables (almost surely they are pairwise distinct), and we let $\pi$ be the ordering of $V$ by increasing $U$-value. For every $v\in V$, let
		\[
		X_v\;=\;\sum_{u\sim v}s_{vu}\,w_{vu}\,\ind{U_u>U_v},
		\]
		where the sum is over the neighbours $u$ of $v$, $w_{vu}$ is the weight of the (unique) arc between $u$ and $v$, and $s_{vu}=+1$ if that arc leaves $v$
		and $s_{vu}=-1$ if it enters $v$. With the notation of
		Lemma~\ref{lem:peel} applied to $\pi$, we have $X_v=w(O_v)-w(I_v)$, and hence
		Lemma~\ref{lem:peel} produces, for every $U$, an
		ordering with surplus at least $\sum_{v\in V}|X_v|$.
		
		Fix $v$ and condition on $U_v=t$. Since $D$ has no directed $2$-cycles, each
		neighbour $u$ of $v$ carries exactly one arc to or from $v$ and hence
		contributes exactly one term to $X_v$, with a fixed sign. In addition, since the
		variables $(U_u)_{u\sim v}$ are independent of each other and of $U_v$, the
		indicators $\bigl(\ind{U_u>t}\bigr)_{u\sim v}$ are independent
		Bernoulli($1-t$) random variables. Hence, let
		$a_v$ be the vector $\bigl(s_{vu}w_{vu}\bigr)_{u\sim v}$ indexed by the neighbours of $v$, whose $\ell_2$-norm is $w_2(v)$ (if $w_2(v)=0$ the bound below is trivial, so assume $w_2(v)>0$). Substituting $q=1-t$ and
		integrating, we have
		\[
		\E\,|X_v|
		\;=\;\int_0^1\E\bigl[\,|X_v|\;\big|\;U_v=t\,\bigr]\,dt
		\;=\;w_2(v)\int_0^1\E\Bigl|\sum_{u\sim v}\tfrac{(a_v)_u}{w_2(v)}\,
		B_u(q)\Bigr|\,dq
		\;\ge\;\bar\kappa\,w_2(v)\;=\;\frac{\sqrt2}{6}\,w_2(v),
		\]
		with $B_u(q)$ i.i.d.\ Bernoulli($q$), by definition \eqref{eq:kbar} of
		$\bar\kappa$ and $\bar\kappa=\frac{\sqrt{2}}{6}$ (proved in Theorem~\ref{thm:kbarmain}). Summing over $v$ and
		using linearity of expectation, there exists a realization of $U$ with
		$\sum_v|X_v|\ge\frac{\sqrt2}{6}\sum_v w_2(v)$, and
		Lemma~\ref{lem:peel} converts it into an ordering $\sigma$ with surplus at
		least this value. This proves the first inequality of the theorem. The second follows from $w_2(v)\ge w(v)/\sqrt{d(v)}\ge w(v)/\sqrt\Delta$ (where we applied Cauchy–Schwarz inequality) and $\sum_v w(v)=2w(D)$, and the bounds on $\fasw(D)$ follow via
		\eqref{eq:ordering}. For the unit-weight directed triangle,
		$\sum_vw_2(v)=3\sqrt2$, $\Delta=2$, $w(D)=3$ and $\fasw=1$, so all four
		inequalities are equalities there. In particular
		$c^*\le\sqrt2\,(\tfrac12-\tfrac13)=\frac{\sqrt2}{6}$, while the bound just
		proven gives $c^*\ge\frac{\sqrt2}{6}$ and therefore $c^*=\frac{\sqrt2}{6}$.
	\end{proof}
	
	\begin{remark}[Algorithmic aspects]\label{rem:algo}
		The proof is effective: sampling the labels $(U_v)_{v\in V}$,
		sorting them, and running the two-block construction of
		Lemma~\ref{lem:peel} takes $O(a(D)+n\log n)$ time, where $a(D)$ is
		the number of arcs, and outputs an ordering meeting the bounds of
		Theorem~\ref{thm:main} in expectation. We do not pursue high-probability guarantees or derandomization here.
	\end{remark}
	
	\section{The integrated Bernoulli--Khintchine constant}\label{sec:kappa}
	
	The engine of Theorem~\ref{thm:main} is a sharp lower bound for the
	integrated first absolute moment of weighted Bernoulli sums. Throughout
	this section and the next, $a\in\R^k$ with $\|a\|_2=1$,
	$B_1(q),\dots,B_k(q)$ are i.i.d.\ Bernoulli($q$),
	$X=X(q)=\sum_{i=1}^ka_iB_i(q)$, $x^+:=\max(x,0)$ denotes the
	positive part, and
	\begin{equation}\label{eq:kbar}
		\bar\kappa\;:=\;\inf_{\norm{a}_2=1}\ \int_0^1
		\E\Bigl|\sum_ja_jB_j(q)\Bigr|\,dq .
	\end{equation}
	Note that when $a=(\frac{1}{\sqrt{2}},-\frac{1}{\sqrt{2}})$, $\int_0^1\E\Bigl|\sum_ja_jB_j(q)\Bigr|\,dq=\frac{\sqrt{2}}{6}$ and therefore 
	
	\[	\bar\kappa\;\leq\; \frac{\sqrt{2}}{6}. \]
	Thus, from now on, we aim to prove $\bar\kappa\ge\frac{\sqrt2}{6}$, splitting the argument according to the shape of the unit vector $a$.
	
	We begin with the characteristic function of $X$. Recall that the characteristic function of a real random variable $Y$ is $\phi_Y(u):=\E e^{iuY}$ ($u\in\R$) (for background on characteristic functions and their use, see, e.g., \cite[Section~3.3]{Dur19}), that $\phi_{cY}(u)=\phi_Y(cu)$ for $c\in\R$ and
	$\phi_Y(-u)=\overline{\phi_Y(u)}$, and that for \emph{independent} $Y$
	and $Z$,
	\begin{equation}\label{eq:mult}
		\phi_{Y+Z}(u)\;=\;\phi_Y(u)\,\phi_Z(u).
	\end{equation}
	For a Bernoulli($q$) variable $B$, we have that
	\[
	\phi(u)\;:=\;\phi_B(u)\;=\;1-q+qe^{iu},\qquad
	r_q(u)\;:=\;|\phi(u)|\;=\;\sqrt{1-4q(1-q)\sin^2(u/2)}.
	\]
	By \eqref{eq:mult} and
	$\phi_{-B'}(u)=\phi_{B'}(-u)=\overline{\phi(u)}$, the square
	$r_q^2=\phi\,\overline{\phi}=\phi_\zeta$ is the (nonnegative)
	characteristic function of $\zeta=B-B'$, the difference of two
	independent Bernoulli($q$) variables: $\zeta$ takes the values $0,\pm1$ with $\Pr(\zeta=1)=\Pr(\zeta=-1)=q(1-q)$. Hence, letting $\rho:=1-2q(1-q)\in[\tfrac12,1]$, we have
	\begin{equation}\label{eq:threepoint}
		\phi_\zeta(u)\;=\;\rho+(1-\rho)\cos u \;=\; r_q^2(u).
	\end{equation}
	For $s\ge1$ set
	\[
	F_q(s)\;=\;\frac2\pi\int_0^\infty\frac{1-r_q(t/\sqrt s)^{\,s}}{t^2}\,dt.
	\]
	We shall use two known results: the Fourier representation of the first absolute moment, due to von~Bahr and Esseen  \cite[Lemma~2]{vBE65}, and an inequality of Havrilla and Tkocz \cite[inequality~(11)]{HT21}. 
	
	\begin{lemma}[\cite{vBE65,HT21}]\label{lem:known}
		\textup{(i)} Every real random variable $Y$ with $\E|Y|<\infty$ and
		characteristic function $\phi_Y$ satisfies
		\[
		\E|Y|\;=\;\frac2\pi\int_0^\infty
		\bigl(1-\operatorname{Re}\phi_Y(t)\bigr)\,\frac{dt}{t^2} .
		\]
		\textup{(ii)} For every $\rho\in[\tfrac12,1]$, the function
		\[
		\mathcal F_\rho(\sigma)\;:=\;\frac2\pi\int_0^\infty
		\Bigl(1-\phi_\zeta(\tau/\sqrt{\sigma})^\sigma\Bigr)\,
		\frac{d\tau}{\tau^2}
		\]
		satisfies $\mathcal F_\rho(\sigma)\ge\mathcal F_\rho(1)$ for every
		$\sigma\ge1$.
	\end{lemma}
	
	\begin{remark}
		In \cite[Lemma~2]{vBE65}, the formula is written as 	$\E|Y|\;=\;\frac1\pi\int_{-\infty}^\infty
		\bigl(1-\operatorname{Re}\phi_Y(t)\bigr)\,\frac{dt}{t^2}$, which is equivalent to part \textup{(i)}, since the integrand is even in $t$: indeed, $\operatorname{Re}\phi_Y(t)=\E\, \cos (tY)$ and cosine is an even function. In \cite[inequality~(11)]{HT21}, $\mathcal F_\rho(\sigma)\ge\mathcal F_\rho(1)$ is stated for $\mathcal F_\rho(\sigma)\;:=\;\frac2\pi\int_0^\infty
		\Bigl(1-|\phi_\zeta(\tau/\sqrt{\sigma})|^\sigma\Bigr)\,
		\frac{d\tau}{\tau^2}$. This implies part \textup{(ii)}: for $\rho\in[\frac{1}{2},1]$, \eqref{eq:threepoint} gives $\phi_\zeta(u)=\rho+(1-\rho)\cos u\ge2\rho-1\ge0$, so $|\phi_\zeta(u)|^\sigma=\phi_\zeta(u)^\sigma$ and the two definitions coincide.
	\end{remark}
	
	The following lemma, which is derived from Lemma~\ref{lem:known}, provides the Fourier-type lower bound we shall work with.
	
	\begin{lemma}\label{lem:fourier}
		Let $a$ be a unit vector in $\R^k$. Then, for every $q\in[0,1]$ and $t_0\in \R$, we have
		\[\ \E|X(q)-t_0|\ \ge\ \sum_{j:a_j\neq 0}a_j^2F_q(a_j^{-2}).\] 
		Moreover, for every $s\geq 2$,
		\[
		F_q(s)\;\ge\;F_q(2)\;=\;\sqrt2\,q(1-q).
		\]
	\end{lemma}
	
	\begin{proof}
		Since $X(q)=\sum_{j}a_jB_j(q)$, by \eqref{eq:mult},  $\phi_X(t)=\prod_j\phi(a_jt)$ and therefore
		\[\operatorname{Re}\phi_{X-t_0}(t)\le|\phi_{X-t_0}(t)|=|\E e^{it(X-t_0)}|=|\E e^{itX}|=|\phi_{X}(t)|=\prod_j|\phi(a_jt)|=\prod_jr_q(a_jt).\] 
		Thus, by Lemma~\ref{lem:known}\textup{(i)},
		\[
		\E|X(q)-t_0|\;\ge\;\frac2\pi\int_0^\infty
		\Bigl(1-\prod_{j:\, a_j\neq 0}r_q(a_jt)\Bigr)\frac{dt}{t^2}.
		\]
		For every $a_j\ne 0$, let $s_j:=a_j^{-2}\ge1$. As $r_q$ is
		an even function, $r_q(a_jt)=r_q(t/\sqrt{s_j})$, and therefore by the weighted AM--GM inequality,
		\[\prod_{j:\, a_j\neq 0}r_q(a_jt)=\prod_{j:\, a_j\neq 0}\bigl(r_q(t/\sqrt{s_j})^{s_j}\bigr)^{a_j^2}
		\le\sum_{j:\, a_j\neq 0}a_j^2\,r_q(t/\sqrt{s_j})^{s_j}.\]
		Thus, as $\|a\|_2=1$, we have that
		
		\begin{equation*}
			\E|X(q)-t_0|\;\ge\;\sum_{j:\, a_j\neq 0}\frac2\pi\int_0^\infty
			\Bigl(a_j^2-a_j^2r_q(t/\sqrt{s_j})^{s_j}\Bigr)\frac{dt}{t^2}
			\;=\;\sum_{j:a_j\neq 0}a_j^2F_q(s_j)\;=\;\sum_{j:a_j\neq 0}a_j^2F_q(a_j^{-2}).
		\end{equation*}

		For the second claim, by \eqref{eq:threepoint}, $r_q(u)^{2}=\phi_\zeta(u)$ and therefore, the substitution $t=\sqrt2\,\tau$ gives
		\[
		F_q(s)\;=\;\frac{\sqrt{2}}{\pi}\int_0^\infty\frac{1-r_q(\tau/\sqrt {s/2})^{\,s}}{\tau^2}\,d\tau\;=\;\frac{\sqrt{2}}{\pi}\int_0^\infty\frac{1-\phi_\zeta(\tau/\sqrt {s/2})^{\,s/2}}{\tau^2}\,d\tau\;=\;\frac{1}{\sqrt{2}}\,\mathcal F_\rho\bigl(s/2\bigr).
		\]
		Hence, for $s\ge2$, Lemma~\ref{lem:known}\textup{(ii)} gives
		$F_q(s)=\frac{1}{\sqrt{2}}\,\mathcal F_\rho\bigl(s/2\bigr)\ge\frac{1}{\sqrt{2}}\mathcal F_\rho(1)=F_q(2)$. In addition, by
		Lemma~\ref{lem:known}\textup{(i)} applied to $\zeta$,
		$\mathcal F_\rho(1)=\E|\zeta|=\Pr(\zeta\ne0)=2q(1-q)$, and therefore
		$F_q(2)=\sqrt2\,q(1-q)$, which completes the proof. 
	\end{proof}
	
	With the above Fourier framework, we have the following, which restricts our discussion to unit vectors $a$ with $\max_j|a_j|>1/\sqrt2$.
	
	\begin{corollary}\label{prop:nondom}
		Let $a$ be any unit vector in $\R^k$ with $\max_j|a_j|\le1/\sqrt2$. Then, for every $t_0\in \R$ and $q\in[0,1]$, we have 
		\[\E|X(q)-t_0|\ge\sqrt2\,q(1-q).\]
		As a result, $\int_0^1\E|X(q)-t_0|\,dq\ge\frac{\sqrt2}{6}$.
	\end{corollary}
	
	\begin{proof}
		As $\max_j|a_j|\le 1/\sqrt2$, $a_j^{-2}\ge2$ for every $j$ with $a_j\neq0$. Thus, by Lemma~\ref{lem:fourier}, $\E|X(q)-t_0|\ge\sum_{j:a_j\neq 0}a_j^2F_q(a_j^{-2})\ge F_q(2)\sum_{j:a_j\neq 0}a_j^2=\sqrt2\,q(1-q)$ and therefore \[\int_0^1\E|X(q)-t_0|\,dq\ge\int_0^1\sqrt2\,q(1-q)\,dq=\frac{\sqrt2}{6},\]
		completing the proof.
	\end{proof}
	
	Before turning to such vectors, we note that the unweighted case is already complete. 
	
	\begin{remark}\label{rem:unweighted}
		For \emph{unweighted} orgraphs (unit weights, $w(D)=|A(D)|$), Theorem~\ref{thm:main} needs none of what follows: at a vertex of degree $d$, the normalized coefficient vector in the proof of Theorem~\ref{thm:main} is $(\pm1,\dots,\pm1)/\sqrt d$, which for $d\ge2$ has no coordinate greater than $1/\sqrt{2}$, so Corollary~\ref{prop:nondom} applies, and for $d=1$ the integral equals $\int_0^1q\,dq=\frac12>\frac{\sqrt{2}}{6}$. The dominated analysis below is
		needed only for genuinely weighted instances, where a single heavy arc produces a dominated coefficient vector at a vertex of arbitrarily large degree. 
	\end{remark}
	
	We call a unit vector $a$ \emph{dominated} if
	$\max_j|a_j|>\frac1{\sqrt2}$. By Corollary~\ref{prop:nondom}, only dominated vectors remain to be treated, and we now give, once and for all, a \emph{standard form} for them. The notation introduced here is kept for the remainder of the paper. A global sign change leaves $\E|X(q)|$ invariant, so we may assume that a dominant coordinate is positive, and we list it first and split the remaining coordinates by sign:
	\[
	a=(\alpha,p_1,\dots,p_r,-c_1,\dots,-c_s),\qquad
	\alpha=\max_j|a_j|>\tfrac1{\sqrt2},
	\]
	with $p_1\ge\dots\ge p_r>0$ and $c_1\ge\dots\ge c_s>0$ (zero coordinates are discarded: they affect neither $\E|X(q)|$ nor $\norm a_2$). Let $B,B_1,\dots,B_r,B_1',\dots,B_s'$ be the corresponding i.i.d. Bernoulli($q$) random variables. Then, we have
	\[
	X\;=\;\alpha B+P-Z,\qquad P\;=\;\sum_{i=1}^rp_iB_i,\qquad
	Z\;=\;\sum_{j=1}^sc_jB_j',
	\]
	and we set
	\[
	R\;=\;\sum_{i=1}^rp_i,\qquad L\;=\;\sum_{j=1}^sc_j,\qquad
	m\;=\;L-R,\qquad \gamma^2\;=\;1-\alpha^2,
	\]
	\[
	J(a)\;=\;\int_0^1\E|X(q)|\,dq,\qquad
	U_0\;=\;\int_0^1(1-q)\,\E(P-Z)^+dq,\qquad
	U_1\;=\;\int_0^1q\,\E(Z-P-\alpha)^+dq.
	\]
	The following result provides an exact formula for $J(a)$, and a lower bound on $U_0$ when $r\ge1$ and $m>0$. Its proof does not use the assumption $\alpha>\frac{1}{\sqrt2}$.
	
	\begin{lemma}\label{lem:mixedid}
		Let $a$ be a dominated unit vector in standard form. Then, \[\displaystyle
		J(a)\;=\;\frac\alpha2-\frac m6+2\,(U_0+U_1).\]
		Furthermore, if $r\ge1$ and $m>0$ then,
		\[
		U_0\;\ge\;(m+p_1)\Bigl(\frac{\theta^{3}}{6}-\frac{\theta^{4}}{12}\Bigr),
		\]
		where $\theta:=\dfrac{p_1}{m+p_1}$.
	\end{lemma}

	\begin{proof}
		Conditioning on $B$ gives
		$\E|X|=(1-q)\,\E|P-Z|+q\,\E|\alpha+P-Z|$. Since $|x|=x+2(-x)^+$
		and $\E(Z-P)=qm$, we have $\E|P-Z|=\E|Z-P|=qm+2\,\E(P-Z)^+$ and
		$\E|\alpha+P-Z|=\alpha-qm+2\,\E(Z-P-\alpha)^+$, so
		\begin{eqnarray*}
			J(a) &=&  \int_0^1\Bigl[(1-q)\bigl(qm+2\,\E(P-Z)^+\bigr)
			+q\bigl(\alpha-qm+2\,\E(Z-P-\alpha)^+\bigr)\Bigr]dq\\
			&=&\alpha\int_0^1 q\,dq+m\int_0^1(q-2q^2)\,dq +2\, (U_0+U_1)\\
			\;&=&\;\frac\alpha2-\frac m6+2\,(U_0+U_1).
		\end{eqnarray*}
		This proves the identity. For the lower bound on $U_0$,
		let $P=p_1B_1+P'$ with $P'=\sum_{i\ge2}p_iB_i$, and note that $B_1$ is independent of the pair $(P',Z)$. Since $(P-Z)^+\ge0$, 
		\[
		\E(P-Z)^+\;\ge\;\E\bigl[(P-Z)^+\;\big|\; B_1=1\bigr]\Prob(B_1=1)
		\;=\;q\,\E\bigl(p_1+P'-Z\bigr)^+.
		\]
		As the map $x\mapsto(p_1-x)^+$ is convex, Jensen's inequality
		applied to the random variable $Z-P'$ yields
		\[
		\E\bigl(p_1+P'-Z\bigr)^+\;=\;\E\bigl(p_1-(Z-P')\bigr)^+
		\;\ge\;\bigl(p_1-\E(Z-P')\bigr)^+,
		\]
		and $\E(Z-P')=\E(Z-P)+p_1\,\E B_1=q\,(m+p_1)$. As $m+p_1>0$, we may
		factor it out of the positive part:
		\[
		\bigl(p_1-q\,(m+p_1)\bigr)^+=(m+p_1)\Bigl(\frac{p_1}{m+p_1}-q\Bigr)^{\!+}
		=(m+p_1)\,(\theta-q)^+ .
		\]
		Combining them, we have $\E(P-Z)^+\ge \,(m+p_1)q(\theta-q)^+$
		for every $q\in[0,1]$. Moreover $\theta\in(0,1)$ because $m>0$, so
		$(\theta-q)^+$ vanishes for $q\ge\theta$ and the definition of
		$U_0$ gives
		\[
		U_0\;\ge\;(m+p_1)\int_0^\theta(1-q)\,q\,(\theta-q)\,dq
		\;=\;(m+p_1)\Bigl(\int_0^\theta q\,(\theta-q)\,dq
		-\int_0^\theta q^2(\theta-q)\,dq\Bigr),
		\]
		which completes the proof as
		$\int_0^\theta q\,(\theta-q)\,dq=\frac{\theta^3}{2}-\frac{\theta^3}{3}
		=\frac{\theta^3}{6}$ and $\int_0^\theta q^2(\theta-q)\,dq=\frac{\theta^4}{3}-\frac{\theta^4}{4}
		=\frac{\theta^4}{12}$.
	\end{proof}
	
	With Lemma \ref{lem:mixedid}, we can now prove the case when $r=0$.
	
	\begin{theorem}\label{thm:onesigned}
		Let $a=(\alpha,-c_1,\dots,-c_s)$ be a unit vector in standard form with $r=0$. Then
		\[
		\int_0^1\E|X(q)|\,dq\;\ge\;\frac{\sqrt2}{6}.
		\]
	\end{theorem}
	
	\begin{proof}
		By Lemma~\ref{lem:mixedid} (as in this case $P=0$ and $R=0$, and therefore $U_0=0$ and $m=L$), we have 
		
		\begin{equation}\label{eq:p0u00}
			\int_0^1\E|X(q)|\,dq\;= \; \frac\alpha2-\frac L6+2U_1.
		\end{equation}
		
		If $L\le3\alpha-\sqrt2$, then as $U_1\geq 0$, \eqref{eq:p0u00} gives $	\int_0^1\E|X(q)|\,dq\ge\frac\alpha2-\frac L6\ge\frac\alpha2-\frac{3\alpha-\sqrt2}{6}
		=\frac{\sqrt2}{6}$ and therefore we are done.
		
		Assume now $L>3\alpha-\sqrt2\ (>\alpha$, as $\alpha>\tfrac1{\sqrt2})$. As $c_1^2\le\sum_jc_j^2=1-\alpha^2=\gamma^2$, we have $c_1\le\gamma<1/\sqrt{2}<\alpha$, and therefore $L>c_1$. On the event $\{B_1=1\}$ we have
		$(Z-\alpha)^+=(c_1+Z'-\alpha)^+$ with $Z'=Z-c_1B_1$, so Jensen's inequality
		for the convex map $x\mapsto x^+$ gives
		\[
		\E(Z-\alpha)^+\;\ge\;q\,\E\bigl(c_1+Z'-\alpha\bigr)^+
		\;\ge\;q\bigl(c_1+qL_1'-\alpha\bigr)^+\;=\;qL_1'\,(q-q_1)^+,
		\]
		where $L_1'=L-c_1>0$ and $q_1=(\alpha-c_1)/L_1'$. Hence,
		\[U_1\ge L_1'\int_0^1q^2(q-q_1)^+dq,\] 
		and we claim
		\begin{equation}\label{eq:qint}
			\int_0^1q^2(q-q_1)^+\,dq\;\ge\;\frac{1-q_1}{12}
			\qquad\text{whenever }q_1\le 3/4 .
		\end{equation}
		Indeed, for $q_1\in[0,1]$ the left side equals
		\[\int_{q_1}^1q^2(q-q_1)\,dq\;=\;\frac{1}{4}-\frac{q_1}{3}-\frac{q_1^4}{4}+\frac{q_1^4}{3}\;=\;\frac{3-4q_1+q_1^4}{12},\]
		and $3-4q_1+q_1^4\geq 1-q_1$ iff $3q_1-q_1^4\le2$, which holds for $q_1\in[0,\tfrac34]$, as there $3q_1-q_1^4$ is increasing (its derivative is $3-4q_1^3>0$) with value $\tfrac{495}{256}<2$ at $q_1=\tfrac34$. On the other hand, for $q_1<0$ the left side
		is $\int_{0}^1q^2(q-q_1)\,dq=\tfrac14-\tfrac{q_1}3\ge\tfrac{1-q_1}{12}$. Thus, if $q_1\leq 3/4$, then by \eqref{eq:qint}, $2U_1\ge\tfrac{L_1'(1-q_1)}{6}=\tfrac{L-\alpha}{6}$, and therefore \eqref{eq:p0u00} gives
		\[	\int_0^1\E|X(q)|\,dq\;\ge\;\frac\alpha2-\frac L6+\frac{L-\alpha}6\;=\;\frac\alpha3\;\ge\;\frac{\sqrt2}6.\]
		
		Thus, it remains to prove $q_1\le 3/4$. As
		$\gamma^2=\sum_jc_j^2\le c_1\sum_jc_j=c_1L$, we have $c_1\ge\gamma^2/L$, and $q_1=(\alpha-c_1)/(L-c_1)$ is non-increasing in $c_1$ on $[0,L)$ (as its derivative in $c_1$ is $(\alpha-L)/(L-c_1)^2<0$), so
		\[
		q_1\;\le\; \frac{\alpha -\gamma^2/L}{L-\gamma^2/L}\;=\; \frac{\alpha L-\gamma^2}{L^2-\gamma^2}\;=:\;\tilde q(L),
		\]
		the denominator being positive since $L>\alpha\ge\gamma$. Moreover, as the numerator of $\tilde q'(L)$
		is $-\alpha L^2+2\gamma^2L-\alpha\gamma^2$, a downward parabola with discriminant $4\gamma^2(\gamma^2-\alpha^2)<0$, $\tilde q'(L)<0$ on $(\gamma,\infty)$ and therefore $\tilde q$
		is strictly decreasing on $(\gamma,\infty)$. Thus, as
		$L\ge3\alpha-\sqrt2>\alpha\geq \gamma$,
		\[
		q_1\;\le\;\tilde q(3\alpha-\sqrt2)
		\;=\;\frac{4\alpha^2-\sqrt2\,\alpha-1}{10\alpha^2-6\sqrt2\,\alpha+1}\;=\;\frac{4(\alpha-\tfrac1{\sqrt2})(\alpha+\tfrac{\sqrt2}4)}{10(\alpha-\tfrac1{\sqrt2})(\alpha-\tfrac{\sqrt2}{10})}
		\;=\;\frac{4\alpha+\sqrt2}{10\alpha-\sqrt2}\;\le\;\frac34 ,
		\]
		where the second last step cancels the common factor $\alpha-\tfrac1{\sqrt2}$ (we can do it as $\alpha>1/\sqrt{2}$) and the last inequality is
		equivalent to $14\alpha\ge7\sqrt2$, i.e. to
		$\alpha\ge\tfrac1{\sqrt2}$. This completes the proof.
	\end{proof}
	
	Thus, from now on we may focus on the case when $r\geq 1$.
	
	We end this section by proving the following unconditional lower bound, which will be useful in the next section. Note that the proofs of \eqref{eq:budget} and of part (i) do not use the assumption $\alpha>1/\sqrt2$.
	
	\begin{corollary}\label{cor:floor}
		Let $\nu:=\frac{\sqrt[3]{36}}{4}-\frac12=0.3254\ldots$,
		$\beta_0:=\frac12\Bigl(\frac{\sqrt2}{6}-\frac{\nu}{\sqrt2}\Bigr)=0.0027\ldots
		<0.003$, and for $m>\alpha$, let
		$g(m):=\frac m2-\frac\alpha2+\frac{\alpha^3}{3m^2}$. Then, every dominated
		unit vector $a$ satisfies
		\begin{equation}\label{eq:budget}
			J(a)\;\ge\;\frac\alpha2-\frac m6+2\int_0^1q\,(qm-\alpha)^+dq+2U_0 ,
		\end{equation}
		and consequently:
		\begin{itemize}
			\item[\textup{(i)}] $J(a)\ge\nu\alpha+2U_0$, and in particular, $J(a)\ge\frac{\sqrt2}{6}$ whenever
			$\alpha\ge\alpha_c:=\frac{\sqrt2}{6\nu} =0.7241\ldots<3/4$.
			\item[\textup{(ii)}] If $m>\alpha\geq 1/\sqrt{2}$, then
			$J(a)\ge\frac{\sqrt2}{6}-2\beta_0+2U_0$ and in particular, $J(a)\ge\frac{\sqrt2}{6}$ if $U_0\ge\beta_0$, or if $m\ge0.87$ and $\alpha\le\alpha_c$.
		\end{itemize}
	\end{corollary}
	
	\begin{proof}
		By Lemma~\ref{lem:mixedid} and Jensen's inequality applied
		to the convex map $x\mapsto x^+$ (recall that $\E(Z-P)=qm$),
		\[U_1\;=\;\int_0^1q\,\E(Z-P-\alpha)^+dq \;\ge\;\int_0^1q\,(\E(Z-P)-\alpha)^+dq\;=\;\int_0^1q\,(qm-\alpha)^+dq, \]
		which gives \eqref{eq:budget}.
		
		For $m\le\alpha$ the integral in \eqref{eq:budget} vanishes, so
		$J(a)\ge\frac\alpha2-\frac m6+2U_0\ge\frac\alpha3+2U_0$. For $m>\alpha$,
		\[
		\int_0^1q\,(qm-\alpha)^+dq\;=\;\int_{\alpha/m}^1q\,(qm-\alpha)\,dq
		\;=\;\frac m3-\frac\alpha2+\frac{\alpha^3}{6m^2},
		\]
		so \eqref{eq:budget} implies
		
		\begin{equation}\label{eq:lbongm}
			J(a)\;\ge\;\frac\alpha2-\frac m6+\frac {2m}{3}-\alpha+\frac{\alpha^3}{3m^2}+2U_0=g(m)+2U_0.
		\end{equation} 
		
		As $g''(m)=\frac{2\alpha^3}{m^4}>0$, the function $g$ is convex on $(\alpha,\infty)$. In addition, as $g'(m)=\frac12-\frac{2\alpha^3}{3m^3}$,
		it has a unique minimizer $m^*=(4/3)^{1/3}\alpha>\alpha$, where $\frac{\alpha^3}{3m^{*2}}=\frac{m^*}{4}$ gives
		\[J(a)\;\geq\; g(m^*)+2U_0\;=\;\frac{3m^*}{4}-\frac\alpha2+2U_0
		\;=\;\bigl(\tfrac34(4/3)^{1/3}-\tfrac12\bigr)\alpha+2U_0\;=\;\nu\alpha+2U_0.\]
		Since $\nu<\frac13$, we have $J(a)\ge\nu\alpha+2U_0$ in both cases. In addition, as $\nu\alpha_c=\frac{\sqrt2}{6}$ and $U_0\geq 0$, $J(a)\ge\frac{\sqrt2}{6}$ whenever $\alpha\ge\alpha_c$, which proves (i).
		
		For (ii), since $\alpha\geq 1/\sqrt{2}$, by (i), we have 
		\[J(a) \;\ge\;\nu\alpha+2U_0\;\ge\;\frac{\nu}{\sqrt2}+2U_0\;=\;\frac{\sqrt2}{6}+2(U_0-\beta_0),\] 
		which proves the first claim, and shows that $U_0\ge\beta_0$ forces $J(a)\ge\frac{\sqrt2}{6}$. Finally, suppose $m\ge0.87$ and $\alpha\le\alpha_c$. Note that then
		$m\ge0.87>\tfrac34>\alpha_c\ge\alpha$, so by \eqref{eq:lbongm} and the fact that $U_0\geq 0$, we have $J(a)\ge g(m)+2U_0\ge g(m)$. We bound $g(m)$ from below in two steps, first
		in $m$ and then in $\alpha$. First, since
		$(4/3)^{1/3}<\frac98$ (indeed,
		$\bigl(\frac98\bigr)^3=\frac{729}{512}>\frac43$), we have
		$m^*=(4/3)^{1/3}\alpha\le\frac98\cdot\frac34=\frac{27}{32}<0.87\le m$. As $g$ is convex on $(\alpha,\infty)$ with minimizer $m^*$, it is nondecreasing on $[m^*,\infty)$, and therefore $g(m)\ge g(0.87)$.
		Second, at $m=0.87$ we regard
		$g(0.87)=\frac{0.87}{2}-\frac\alpha2+\frac{\alpha^3}{3\cdot0.87^2}$ as a function of $\alpha$, and its derivative
		$-\frac12+\frac{\alpha^2}{0.87^2}\ge-\frac12+\frac{1/2}{0.87^2}>0$ for $\alpha\ge\frac1{\sqrt2}$. Thus, on
		$[\frac1{\sqrt2},\alpha_c]$ it is minimized at
		$\alpha=\frac1{\sqrt2}$ and therefore
		\[
		g(0.87)\;\geq \;0.435-\frac{\sqrt2}{4}+\frac{\sqrt2}{12\cdot0.87^2}
		\;\ge\;0.435-0.3536+0.1557\;=\;0.2371\;>\;\frac{\sqrt2}{6}.
		\]
		Hence $J(a)\ge g(m)>\frac{\sqrt2}{6}$, completing the proof of (ii).
	\end{proof}

	\section{Resolution of the residual vectors: $\bar\kappa=\sqrt2/6$}
	\label{sec:residual}
	
	Throughout this section $a$ is a \emph{residual} vector: a
	dominated unit vector in standard form which is covered by none of
	Theorem~\ref{thm:onesigned} and Corollary~\ref{cor:floor}(i).  Explicitly,
	\[
	r\;\ge\;1\qquad\text{and}\qquad
	\alpha\;\in\;\bigl[\tfrac1{\sqrt2},\,\alpha_c\bigr]\subseteq\bigl[\tfrac{1}{\sqrt{2}},\tfrac{3}{4}\bigr].
	\]
	Indeed, if $\max_j|a_j|\le\frac1{\sqrt2}$ then
	Corollary~\ref{prop:nondom} (with $t_0=0$) gives
	$J(a)\ge\frac{\sqrt2}{6}$. Theorem~\ref{thm:onesigned} covers
	dominated vectors with $r=0$, and Corollary~\ref{cor:floor}(i)
	covers $\alpha\ge\alpha_c$. Hence, in order to prove
	$\bar\kappa\ge\frac{\sqrt2}{6}$, it suffices to show
	$J(a)\ge\frac{\sqrt2}{6}$ for every residual vector.
	We may and do further assume
	\begin{equation}\label{eq:mwindow}
		m\;>\;3\alpha-\sqrt2\;\;(\geq \alpha),
	\end{equation}
	since otherwise Lemma~\ref{lem:mixedid} with $U_0,U_1\ge0$ gives $J(a)\ge\frac\alpha2-\frac m6\ge\frac{\sqrt2}{6}$ directly.
	
	We next collect all further notation of this section. Recall that by the convention of the standard form, 
	
	\[p_1\ge\dots\ge p_r>0\qquad \text{and} \qquad c_1\ge\dots\ge c_s>0\]
	so that $c_2$ denotes the second largest of the $c_j$ ($c_2:=0$ if $s\leq 1$). Denote
	\[
	d=\alpha-c_1,\qquad m_2=m-c_1,\qquad q_1=\frac{d}{m_2},\qquad
	\tau=2\alpha-\sqrt2,\qquad \delta=\gamma-c_1,
	\]
	and $\gamma_c^2=\gamma^2-c_1^2=\delta(2\gamma-\delta)$, the squared
	$\ell_2$-norm of the coefficient vector of
	\[
	W\;:=\;\sum_{j\ge2}c_jB_j'-P,
	\]
	the sum of all Bernoulli terms of $Z-P$ except that of $c_1$.
	As $\alpha\geq \frac{1}{\sqrt{2}}$, $c_1\le\gamma\le\tfrac1{\sqrt2}\le\alpha$, and $p_1>0$
	forces $c_1<\gamma$. Hence $d>0$ and $\delta>0$, while
	\eqref{eq:mwindow} gives $m_2>d$, so $0<q_1<1$. Further put
	\[
	\psi(t)\;:=\;\frac t2-\frac{t^4}{6}-\frac13,
	\]
	and \[
	I_3\;:=\;\int_{1/\sqrt2}^1q^3(1-q)\,dq=\frac{2\sqrt2-1}{80}, \qquad
	I_2\;:=\;\int_{1/\sqrt2}^1q^2(1-q)\,dq=\frac{7-4\sqrt2}{48}.
	\]
	Finally, we call a function $\Phi\colon[0,1]\to[0,\infty)$ a
	\emph{floor} if $\E|W-t|\ge\Phi(q)$ for every $t\in\R$ and every
	$q\in[0,1]$.
	
	The section provides two tools --- a fluctuation floor lemma and a fluctuation ledger --- which the proof of Theorem~\ref{thm:kbarmain} then combines with the bound of Corollary~\ref{cor:floor}(ii).

	\begin{lemma}\label{lem:ufloor}
		Let $a$ be a residual vector. Then the following hold.
		\begin{itemize}
			\item[\textup{(i)}] $\Phi(q)=\dfrac{\gamma_c}{\sqrt2}\,\min(q,1-q)$
			is a floor.
			\item[\textup{(ii)}] If $\max(p_1,c_2)\le\gamma_c/\sqrt2$, then
			$\Phi(q)=\sqrt2\,q(1-q)\,\gamma_c$ is a floor.
		\end{itemize}
	\end{lemma}
	
	\begin{proof}
		Denote by $u_1,\dots,u_n$ (with $n=r+s-1$) the coefficients
		$c_2,\dots,c_s,-p_1,\dots,-p_r$ of $W$, and
		$\beta_1,\dots,\beta_n$ for the corresponding independent
		Bernoulli($q$) variables, so that $W=\sum_ju_j\beta_j$ and
		$\sum_ju_j^2=\gamma_c^2$, with $\gamma_c>0$ since $\gamma>c_1$ (discussed in the section preamble). Fix $t\in\R$ and $q\in[0,1]$; we now bound $\E|W-t|$ from
		below in two complementary situations.
		
		Suppose first that $\max(p_1,c_2)=\max_j |u_j|\le\gamma_c/\sqrt2$. Then
		$u/\gamma_c$ is a unit vector with no coordinate exceeding
		$\frac1{\sqrt2}$ in absolute value, so
		Corollary~\ref{prop:nondom}, applied to $u/\gamma_c$ with the
		shift $t_0=t/\gamma_c$, gives
		\[
		\E|W-t|\;=\;\gamma_c\,
		\E\Bigl|\sum_j\frac{u_j}{\gamma_c}\,\beta_j-\frac t{\gamma_c}\Bigr|
		\;\ge\;\sqrt2\,q(1-q)\,\gamma_c
		\;\ge\;\frac{\gamma_c}{\sqrt2}\,\min(q,1-q),
		\]
		where the last inequality holds as $2q(1-q)\ge\min(q,1-q)$: for $q\le\frac12$ one has $2q(1-q)\ge2q\cdot\frac12=q$, and symmetrically for $q\ge\frac12$. In particular, this completes the proof of (ii), and also proves (i) in this situation. It remains to prove (i) in the complementary situation.
		
		Suppose now that there is an index $j^*$ such that $|u_{j^*}|>\gamma_c/\sqrt2$. Condition on $(\beta_j)_{j\ne j^*}$: the
		conditional expectation of $|W-t|$ equals
		$\E|u_{j^*}\beta_{j^*}-t'|$, where
		$t'=t-\sum_{j\ne j^*}u_j\beta_j$ is a constant under the
		conditioning. For all $c,t'\in\R$, the triangle inequality
		$|t'|+|c-t'|\ge|c|$ gives
		\[
		\E|c\,\beta_{j^*}-t'|\;=\;(1-q)\,|t'|+q\,|c-t'|
		\;\ge\;\min(q,1-q)\,\bigl(|t'|+|c-t'|\bigr)
		\;\ge\;\min(q,1-q)\,|c| .
		\]
		With $c=u_{j^*}$ this is at least
		$\frac{\gamma_c}{\sqrt2}\min(q,1-q)$. Thus, taking the expectation over the conditioning variables completes the proof of (i).
	\end{proof}

	\begin{lemma}\label{lem:ledger}
		Let $a$ be residual and let $\Phi$ be any floor. Then, with $\psi$, $\tau$, $m_2$ and $q_1$ as in the section preamble,
		\[
		J(a)-\frac{\sqrt2}{6}\;\ge\;\frac\tau6-m_2\,\psi(q_1)
		+\int_0^1q^2\bigl(\Phi(q)-m_2|q-q_1|\bigr)^+dq .
		\]
	\end{lemma}
	
	\begin{proof}
		Write $V:=Z-P$, so that $V=c_1B_1'+W$ with $B_1'$
		independent of $W$, and recall from the preamble that
		$\E W=q(m-c_1)=qm_2$, that $m_2q_1=d=\alpha-c_1$, and that
		$U_1=\int_0^1q\,\E(V-\alpha)^+dq$.
		
		We first bound $\E|V-\alpha|$ from below. Conditioning on
		$B_1'$,
		\[
		\E|V-\alpha|\;=\;(1-q)\,\E|W-\alpha|+q\,\E|W-d| .
		\]
		For the first term, as $|W-\alpha|\ge\alpha-W$,
		$\E|W-\alpha|\ge\alpha-qm_2$. The second term admits two lower
		bounds simultaneously: Jensen's inequality gives
		$\E|W-d|\ge|\E W-d|=|qm_2-d|=m_2|q-q_1|$, and the definition
		of a floor gives $\E|W-d|\ge\Phi(q)$. Hence, by
		$\max(y,\Phi)=y+(\Phi-y)^+$,
		\[
		\E|W-d|\;\ge\;\max\bigl(m_2|q-q_1|,\Phi(q)\bigr)
		\;=\;m_2|q-q_1|+\bigl(\Phi(q)-m_2|q-q_1|\bigr)^+,
		\]
		and altogether
		\[
		\E|V-\alpha|\;\ge\;(1-q)(\alpha-qm_2)+q\,m_2|q-q_1|
		+q\bigl(\Phi(q)-m_2|q-q_1|\bigr)^+ .
		\]
		
		Next we pass from the absolute value to the positive part.
		Since $x^+=\frac12(|x|+x)$ and $\E(V-\alpha)=qm-\alpha$,
		\[
		\E(V-\alpha)^+\;=\;\tfrac12\bigl(\E|V-\alpha|+qm-\alpha\bigr).
		\]
		From $m_2q_1=\alpha-c_1$ we obtain the two identities
		\[
		qm-\alpha\;=\;m_2(q-q_1)-c_1(1-q),\qquad
		(1-q)(\alpha-qm_2)\;=\;(1-q)\,m_2(q_1-q)+c_1(1-q),
		\]
		and when they are added the terms $\pm c_1(1-q)$ cancel
		exactly, and thus we have
		\begin{eqnarray*}
			(1-q)(\alpha-qm_2)+q\,m_2|q-q_1|+qm-\alpha
			\;&=&\;m_2\bigl[(1-q)(q_1-q)+q|q-q_1|+(q-q_1)\bigr]\\
			\;&=&\;m_2\,q\,\big[(q-q_1)+|q-q_1|\bigr]\;=\;2q\,m_2(q-q_1)^+ .
		\end{eqnarray*}
		Combining all of them, we have
		\[
		\E(V-\alpha)^+\;\ge\;q\,m_2(q-q_1)^+
		+\tfrac q2\bigl(\Phi(q)-m_2|q-q_1|\bigr)^+ .
		\]
		
		Multiplying by $2q$ and integrating over $q\in[0,1]$ yields
		\[
		2U_1\;\ge\;2m_2h(q_1)
		+\int_0^1q^2\bigl(\Phi(q)-m_2|q-q_1|\bigr)^+dq,
		\]
		where 
		$h(q_1):=\int_{q_1}^1q^2(q-q_1)\,dq
		=\frac14-\frac{q_1}3+\frac{q_1^4}{12}$ (note that as $0<q_1<1$, the integral of $q^2(q-q_1)^+$ over $[0,1]$ is
		indeed $h(q_1)$).
		
		Finally we assemble. By Lemma~\ref{lem:mixedid} and the fact that $U_0\ge0$,
		\begin{eqnarray*}
			J(a)-\frac{\sqrt2}{6}\;&\ge&\;\frac\alpha2-\frac m6
			-\frac{\sqrt2}{6}+2U_1\\
			\;&\geq& \;\frac{3\alpha-m-\sqrt2}{6}+2m_2h(q_1)
			+\int_0^1q^2\bigl(\Phi(q)-m_2|q-q_1|\bigr)^+dq\\
			\;&=&\;\frac{\tau+\alpha-m}{6}+2m_2h(q_1)+\int_0^1q^2\bigl(\Phi(q)-m_2|q-q_1|\bigr)^+dq\\
			\;&=&\;\frac{\tau-m_2(1-q_1)}{6}+2m_2h(q_1)+\int_0^1q^2\bigl(\Phi(q)-m_2|q-q_1|\bigr)^+dq\\
			\;&=&\;\frac{\tau}{6}-m_2\psi(q_1)+\int_0^1q^2\bigl(\Phi(q)-m_2|q-q_1|\bigr)^+dq,
		\end{eqnarray*}
		where the second last equality uses $m_2(1-q_1)=m_2-d=m-\alpha$, and the last one uses
		$\frac{1-q_1}{6}-2h(q_1)=\frac{q_1}2-\frac{q_1^4}{6}-\frac13=\psi(q_1)$.
		This completes the proof.
	\end{proof}
	
	The deficit $m_2\psi(q_1)$ and the drift term of the fluctuation
	integral combine into an affine function of $q_1$ --- the quartic
	terms cancel exactly --- and this is what closes the proof of the main
	theorem below.

	\begin{theorem}\label{thm:kbarmain}
		$\bar\kappa=\dfrac{\sqrt2}{6}$. 
	\end{theorem}
	
	\begin{proof}
		As explained at the beginning of this section, it suffices to show $J(a)\ge\frac{\sqrt2}{6}$ for every residual vector $a=(\alpha,p_1,\dots,p_r,-c_1,\dots,-c_s)$. Also as explained there, we may assume that $a$ satisfies \eqref{eq:mwindow}. Moreover, since $m>\alpha$ and $\alpha\le\alpha_c$, Corollary~\ref{cor:floor}(ii) allows us to further assume $m<0.87$.
		
		As $r\geq 1$, $p_1>0$. Suppose first that $p_1\ge\frac13$. Then, by Lemma~\ref{lem:mixedid}, we have
		\[U_0\;\ge\;(m+p_1)\Bigl(\frac{\theta^3}{6}-\frac{\theta^4}{12}\Bigr)\;=\;\frac{p_1}{12}\,\theta^2(2-\theta),
		\qquad \theta=\frac{p_1}{m+p_1}.\]
		Since $\frac{d}{d\theta}\,\theta^2(2-\theta)=\theta(4-3\theta)>0$ on
		$(0,1)$, the lower bound is increasing in $\theta=\frac{p_1}{m+p_1}$.  As $m<0.87$ and $p_1\geq 1/3$, $\theta\geq \frac{1/3}{0.87+1/3}\geq  \frac{1/3}{9/10+1/3}=\frac{10}{37}$ and therefore
		
		\[U_0\ge\frac1{36}\,\theta^2(2-\theta)\geq \frac1{36}\,\left(\frac{10}{37}\right)^2\left(\frac{64}{37}\right) =\frac{1600}{9\cdot 37^3}\ge\frac{1}{9\cdot 37}>0.003\geq \;\beta_0;\]
		Corollary~\ref{cor:floor}(ii) then gives $J(a)\ge\frac{\sqrt2}{6}$.
		We may therefore assume $p_1<\frac13$.
		
		{\bf Case 1: $q_1\le\frac1{\sqrt2}$.} Since
		$\psi'(t)=\frac12-\frac23t^3>0$ for $t^3<\frac34$, $\psi$ increases on $[0,\tfrac1{\sqrt2}]$. Thus,
		$\psi(q_1)\le\psi(\tfrac1{\sqrt2})=\frac{2\sqrt2-3}{8}<0$, and as $m_2>0$, Lemma~\ref{lem:ledger} with the floor $\Phi=0$ gives
		
		\[J(a)-\frac{\sqrt2}{6}\ge\frac\tau6+m_2\,\frac{3-2\sqrt2}{8}\;>\;0,\]
		as $\tau\ge0$ and $m_2>0$. This completes the proof of Case 1.
		
		Thus, from now on we assume that $q_1>\frac1{\sqrt2}$. Recall
		$I_3$ and $I_2$ from the section preamble. Discarding the positive
		part of the integrand on $[\tfrac1{\sqrt2},1]$ (as $y^+\ge y$) and
		the integrand on the complement (as $y^+\ge0$),
		Lemma~\ref{lem:ledger} gives, for any floor $\Phi$,
		\begin{eqnarray*}
			J-\frac{\sqrt2}{6}\;&\ge&\;\frac\tau6-m_2\,\psi(q_1)
			+\int_{1/\sqrt{2}}^1q^2\bigl(\Phi(q)-m_2|q-q_1|\bigr)dq\\
			\;&\ge&\;\frac\tau6
			+\int_{1/\sqrt{2}}^1q^2\Phi(q)dq-m_2\left(\int_{1/\sqrt{2}}^1q^2|q-q_1|dq+\psi(q_1)\right)\\
			\; &=& \;\frac\tau6
			+\int_{1/\sqrt{2}}^1q^2\Phi(q)dq-m_2\left(\int_{q_1}^1q^2(q-q_1)dq+\int^{q_1}_{1/\sqrt{2}}q^2(q_1-q)dq+\psi(q_1)\right)\\
			\; &=& \;\frac\tau6
			+\int_{1/\sqrt{2}}^1q^2\Phi(q)dq\\
			&&{}-m_2\left(\left(\frac{1}{4}-\frac{q_1}{3}+\frac{q_1^4}{12}\right)+\left(\frac{q_1^4}{12}+\frac{1}{16}-\frac{\sqrt{2}q_1}{12}\right)+\left(\frac{q_1}{2}-\frac{q_1^4}{6}-\frac13\right)\right)\\
			\; &=& \;\frac\tau6
			+\int_{1/\sqrt{2}}^1q^2\Phi(q)dq-m_2q_1\left(\frac16-\frac{\sqrt2}{12}\right)+\frac{m_2}{48}\\
			\;&\ge&\;\frac\tau6+\int_{1/\sqrt2}^1q^2\,\Phi(q)\,dq
			-I_2\,d,
		\end{eqnarray*}
		where the third line splits the integral at $q_1$ (recall
		$\frac1{\sqrt2}<q_1<1$), the fourth line evaluates the three terms,
		the fifth line collects them (the quartic terms cancel), and the
		last inequality uses $m_2q_1=d$, $m_2\ge d$ and
		$\frac16-\frac{\sqrt2}{12}-\frac1{48}=\frac{7-4\sqrt2}{48}=I_2$.
		We complement this bound with the following estimate of its drift
		term.
		
		\begin{equation}\label{eq:drift}
			\dfrac\tau6-I_2\,d\;\ge\;-I_2\,\delta.
		\end{equation}
		
		Indeed, the function $\alpha\mapsto\alpha-\gamma=\alpha-\sqrt{1-\alpha^2}$
		vanishes at $\alpha=\frac1{\sqrt2}$, and its derivative
		$1+\frac\alpha\gamma$ is increasing in $\alpha$ (as $\gamma$
		decreases in $\alpha$). Thus, as  $\alpha\leq \alpha_c\leq 3/4$ it is
		at most $1+\frac{3/4}{\sqrt{1-(3/4)^2}}=1+\frac{3}{\sqrt{7}}\leq \frac{5}{2}$. By the mean value theorem,
		$\alpha-\gamma\le\frac{5}{2}\bigl(\alpha-\frac1{\sqrt2}\bigr) =\frac{5}{4}\,\tau$, and therefore $d-\delta=\alpha-\gamma\le \frac{5}{4}\,\tau$. Since
		$\frac{5}{4}\,I_2\leq 2 I_2=\frac{7-4\sqrt{2}}{24}<\frac16$ and $\tau\ge0$,
		\[
		I_2(d-\delta)\;\leq\; \frac{5}{4}\tau I_2 \;\le\;\frac{\tau}{6},
		\]
		which rearranges to \eqref{eq:drift}. Combining the displayed chain with \eqref{eq:drift}, for every floor
		$\Phi$, we have
		\begin{equation}\label{eq:handcore}
			J(a)-\frac{\sqrt2}{6}\;\ge\;
			\int_{1/\sqrt2}^1q^2\,\Phi(q)\,dq\;-\;I_2\,\delta .
		\end{equation}
		
		The two remaining cases apply \eqref{eq:handcore} with the two floors provided by Lemma~\ref{lem:ufloor}.
		
		{\bf Case 2: $q_1>\frac1{\sqrt2}$ and
			$\max(p_1,c_2)\le\gamma_c/\sqrt2$.} By
		Lemma~\ref{lem:ufloor}(ii), $\Phi=\sqrt2\,q(1-q)\,\gamma_c$ is a
		floor, and
		\[
		\int_{1/\sqrt2}^1q^2\,\Phi(q)\,dq
		\;=\;\sqrt2\,\gamma_c\int_{1/\sqrt2}^1q^3(1-q)\,dq
		\;=\;\sqrt2\,\gamma_c\,I_3 ,
		\]
		so \eqref{eq:handcore} gives
		$J(a)-\frac{\sqrt2}{6}\ge\sqrt2\,\gamma_c\,I_3-I_2\,\delta$. We
		show that $\sqrt2\,\gamma_c\,I_3\ge I_2\,\delta$, which completes
		this case. Since $0<\delta\le\gamma$ (as $\delta=\gamma-c_1$ and
		$0<c_1<\gamma$), we have
		$\gamma_c^2=\delta(2\gamma-\delta)\ge\delta\gamma$, so
		$\sqrt2\,\gamma_c\,I_3\ge\sqrt{2\gamma\delta}\,I_3$, and after
		squaring it suffices to check
		$2\gamma\delta\,I_3^2\ge I_2^2\,\delta^2$, i.e.\
		$2\gamma\,I_3^2\ge I_2^2\,\delta$. As $\delta\le\gamma$, this
		follows from $2I_3^2\ge I_2^2$, that is, from
		$\sqrt2\,I_3\ge I_2$. To verify the latter, 
		\[
		\sqrt2\,\frac{2\sqrt2-1}{80}\;\ge\;\frac{7-4\sqrt2}{48}
		\iff 48\,(4-\sqrt2)\;\ge\;80\,(7-4\sqrt2)
		\iff 17\sqrt2\;\ge\;23 ,
		\]
		and the last inequality holds since $578\ge529$.

		{\bf Case 3: $q_1>\frac1{\sqrt2}$ and
			$\max(p_1,c_2)>\gamma_c/\sqrt2$.} We first show that
		$\delta<\frac{2\gamma}3$. 
		
		Suppose $p_1>\gamma_c/\sqrt2$. Then $\gamma_c^2<2p_1^2$, and
		$p_1<\frac13$ gives $\gamma_c^2<\frac29$. The map
		$\delta\mapsto\delta(2\gamma-\delta)$ is increasing on
		$[0,\gamma]$ (its derivative is $2(\gamma-\delta)\ge0$), and its
		value at $\delta=\frac{2\gamma}3$ is
		$\frac{2\gamma}3\cdot\frac{4\gamma}3=\frac{8\gamma^2}{9}$. Since
		$\alpha\le\alpha_c\leq 3/4$, we have $\gamma^2=1-\alpha^2\ge1-(3/4)^2=7/16$,
		so $\frac{8\gamma^2}{9}\ge\frac{7}{18}
		>\frac29$. As $\gamma_c^2=\delta(2\gamma-\delta)<\frac29$ and
		$\delta\le\gamma$, monotonicity forces $\delta<\frac{2\gamma}3$.
		
		Suppose instead $c_2>\gamma_c/\sqrt2$. Then
		$c_1\ge c_2>\gamma_c/\sqrt2$, so $c_1^2>\gamma_c^2/2$, i.e.\
		$(\gamma-\delta)^2>\delta(2\gamma-\delta)/2$. Expanding,
		$2\gamma^2-4\gamma\delta+2\delta^2>2\gamma\delta-\delta^2$, i.e.\
		$3\delta^2-6\gamma\delta+2\gamma^2>0$. The roots of
		$3x^2-6\gamma x+2\gamma^2$ are
		$x=\gamma\bigl(1\pm\tfrac1{\sqrt3}\bigr)$, so positivity together
		with $\delta\le\gamma<\gamma\bigl(1+\tfrac1{\sqrt3}\bigr)$ forces
		$\delta<\gamma\bigl(1-\tfrac1{\sqrt3}\bigr)<\frac{2\gamma}3$.
		
		Now, by Lemma~\ref{lem:ufloor}(i),
		$\Phi=\frac{\gamma_c}{\sqrt2}\min(q,1-q)$ is a floor. On
		$[\tfrac1{\sqrt2},1]$ we have $\min(q,1-q)=1-q$ (as
		$q\ge\tfrac1{\sqrt2}>\tfrac12$), so
		\[
		\int_{1/\sqrt2}^1q^2\,\Phi(q)\,dq
		\;=\;\frac{\gamma_c}{\sqrt2}\int_{1/\sqrt2}^1q^2(1-q)\,dq
		\;=\;\frac{\gamma_c}{\sqrt2}\,I_2 ,
		\]
		and \eqref{eq:handcore} gives
		\[
		J(a)-\frac{\sqrt2}{6}
		\;\ge\;\Bigl(\frac{\gamma_c}{\sqrt2}-\delta\Bigr)I_2 .
		\]
		Finally, $\frac{\gamma_c}{\sqrt2}>\delta$ is equivalent to
		$\gamma_c^2>2\delta^2$, i.e.\ to
		$\delta(2\gamma-\delta)>2\delta^2$; dividing by $\delta>0$, this is
		$2\gamma-\delta>2\delta$, i.e.\ exactly $\delta<\frac{2\gamma}3$,
		which was shown above. Hence $J(a)>\frac{\sqrt2}{6}$ in this case.
		
		This exhausts all cases, so $\bar\kappa\ge\frac{\sqrt2}{6}$.
		Together with
		$\bar\kappa\le\frac{\sqrt2}{6}$ (the pair computation displayed
		after \eqref{eq:kbar}) this gives
		$\bar\kappa=\frac{\sqrt2}{6}$.
	\end{proof}
	
	\section{Concluding remarks}\label{sec:conclusion}
	
	Theorem~\ref{thm:main} determines $c^*=\tfrac{\sqrt2}{6}$. By the
	theorem, the infimum in \eqref{eq:cstar} is attained, at $\Delta=2$, by
	the unit-weight directed triangle. One problem remains open. Our methods evaluate the
	integrated functional \eqref{eq:kbar} exactly, but leave open its
	pointwise analogue: determine, for fixed $q\in(0,1)$,
	\[
	\kappa(q)\;:=\;\inf_{\norm a_2=1}\E\Bigl|\sum\nolimits_ja_jB_j(q)\Bigr| .
	\]
	Two test vectors give upper bounds. The singleton $a=(1)$ gives
	$\E|B(q)|=q$, and the pair $a=(\tfrac1{\sqrt2},-\tfrac1{\sqrt2})$
	gives $\frac1{\sqrt2}\,\E|B_1(q)-B_2(q)|=\sqrt2\,q(1-q)$, as
	$B_1-B_2$ equals $\pm1$ with probability $2q(1-q)$ and $0$
	otherwise. Hence
	$\kappa(q)\le\min\bigl(q,\sqrt2\,q(1-q)\bigr)$, the singleton
	branch being the smaller one for $q\le1-\tfrac1{\sqrt2}$ and the
	pair branch beyond. We conjecture that equality holds for all
	$q\in(0,1)$. For $q=\tfrac12$ this follows from Szarek's
	inequality \cite{Sza76} (see also \cite{Haa81}): writing $B_j=\frac{1+\varepsilon_j}{2}$ with
	random signs $\varepsilon_j$ (i.e. $\Prob(\varepsilon_j=+1)=\Prob(\varepsilon_j=-1)=1/2$), we have
	$\E|\sum_ja_jB_j(\tfrac12)|=\tfrac12\,\E|s+\sum_ja_j\varepsilon_j|$ with $s=\sum_ja_j$, while a shift can only increase the first absolute moment of a symmetric random variable, and
	$\inf_{\norm a_2=1}\E|\sum_ja_j\varepsilon_j|=\frac1{\sqrt2}$
	\cite{Sza76}, whence $\kappa(\tfrac12)=\frac{\sqrt2}{4}$, matching the pair branch.
	
	Whatever the answer, the pointwise and integrated problems are
	quantitatively separated: integrating the upper bound,
	\[
	\int_0^1\kappa(q)\,dq\;\le\;
	\int_0^1\min\bigl(q,\sqrt2\,q(1-q)\bigr)\,dq
	\;=\;\frac{7-3\sqrt2}{12}
	\;<\;\frac{\sqrt2}{6}\;=\;\bar\kappa,
	\]
	Thus the infimum of the integral strictly exceeds the integral of the pointwise infimum: no single unit vector is optimal at every $q$ simultaneously --- the pointwise minimizer must change with $q$ --- whereas the integrated problem is solved by the pair alone. This is why our evaluation of $\bar\kappa$ could not proceed from a pointwise bound, and required the global analysis of Sections~\ref{sec:kappa} and~\ref{sec:residual}.

\end{document}